\documentclass{article}
\usepackage{amsmath, amsthm, amssymb}
\usepackage[margin=3.5cm]{geometry}
\usepackage{hyperref} %
\usepackage{verbatim}
\usepackage{url}
\usepackage{yfonts}
\usepackage{youngtab}
\usepackage{shuffle}

\usepackage{tikz, xcolor, mathrsfs, multicol}
\usetikzlibrary{shapes.geometric, arrows}

\usepackage{enumerate}

\newtheorem{thm}{Theorem}[subsection]
\newtheorem{prop}[thm]{Proposition}
\newtheorem{lem}[thm]{Lemma}

\theoremstyle{remark}
\newtheorem{rem}[thm]{Remark}

\newtheorem{exm}[thm]{Example}

\title{Counting occurrences of a pattern in a binary word}
\author{Roger Tian\footnote{\href{mailto:rgtian@ucdavis.edu}{rgtian@ucdavis.edu} or \href{mailto:htrland@gmail.com}{htrland@gmail.com}}}
\date{\today}

\begin{document}
\maketitle

\begin{abstract}
Enumerating the number of times one word occurs in another is a much-studied combinatorial subject. By utilizing a method that we call ``lexicographic extreme referencing'', we provide a formula for computing occurrences of one binary word in another. We then study $B_{n,p}(k)$, the number of binary words of length $n$ containing a given word $p$ exactly $k$ times. For this purpose, we first use lexicographic extreme referencing to provide an algorithm for constructing all words $w$ that contain a given word $p$. Afterward, we give a modified version of this algorithm for constructing the subset of binary words that are ``primitive'' with respect to $p$, and we discuss approaches for finding $B_{n,p}(k)$ via primitive words.
\end{abstract}

\section{Introduction}
Enumeration of patterns in words is a major subject of research in combinatorics. One can read about the rich body of work already done on pattern-avoidance in permutations alone \cite{Kit11} \cite{Vat15}. The notion of patterns in words depends on the problem under consideration. In this paper, we take patterns to be subwords of a given word that match another fixed word exactly; a subword is formed by deleting any set of letters in the original word.

One major direction of research in this subject is to study how many times a pattern $p$ occurs as a subword of another word $w$. Statistical results such as mean, variance, and a Gaussian limit law have been obtained for the number of occurrences of a pattern in a random word \cite{Flaj}. In connection with algorithmics and data mining, researchers have also given algorithms to find patterns that occur numerous times \cite{IITN05}. In \cite{Fang}, a concept of subword entropy was formulated to analyze the maximal number of occurrences among all subwords. As of the time of writing, we are unaware of any formula in the literature for computing the exact number of occurrences of a given pattern in a given word.

In the opposite direction, work has also been done on counting words that contain a given pattern a fixed number of times \cite{BurMan} \cite{HeuMan} \cite{Menon}. We focus on the case of binary words (where the letters are 0 or 1) in this paper. Subwords in binary words have proven useful in studying permutation patterns \cite{GilTom} \cite{HongLi} \cite{Meno2} \cite{Meno1}. In \cite{Menon}, the authors studied the number $B_{n,p}(k)$ of binary words having length $n$ that contain a given pattern $p$ exactly $k$ times. They found formulas for $B_{n,p}(k)$ for a few small values of $k$. However, finding $B_{n,p}(k)$ becomes much more difficult for larger $k$, due to the many more ways occurrences can sum up to $k$, hence it seems that a different approach is needed to tackle the general case.

This paper studies the aforementioned problems for binary words and is organized as follows. In Section \ref{calc occur}, we introduce a method that we dub ``lexicographic extreme referencing'' and use it to obtain a formula (in Theorem \ref{count form}) for counting the number of occurrences of a given pattern in a given binary word. Then we use related concepts to give an algorithm for constructing all words that contain a given pattern, in Section \ref{build word}. Finally, in Section \ref{primitive}, we define a subset of binary words that are ``primitive'' and we modify the aforementioned algorithm to give one for constructing them, then we discuss approaches for finding $B_{n,p}(k)$ via primitive words.

\section{Acknowledgments}
The author would like to thank Anurag Singh for helpful discussions via email.

\section{Preliminaries}
We will mostly follow the notations and definitions in \cite{Menon}. A word $w$ is a finite sequence $w = w_1w_2 \ldots w_n$ where $w_i \in \{0, 1, 2, \ldots l\}$ for all $i \in [n]$. The $w_i$ are the \textbf{letters} of $w$, and $\{0, 1, 2, \ldots l\}$ is the \textbf{alphabet}. We denote the \textbf{length} of $w$ by $|w| = n$. We identify any subword $w'$ of $w$ by the tuple of indices locating $w'$ in $w$. For a binary word $p = p_1p_2 \ldots p_m$, an \textbf{occurrence} of $p$ in $w$ is a subword of $w$ that equals $p$; here we also refer to $p$ as a \textit{pattern}. Denote by $c_p(w)$ the number of occurrences of $p$ in $w$.

% If $u, v$ are binary words that both start with a 0 or both start with a 1, we say that they \textbf{start equally}. 
Since we will be studying binary words, the alphabet will be $\{0,1\}$. For a binary word $w$, a \textbf{run} of $w$ is a maximal subword consisting of consecutive letters that are identical (i.e. all 0 or all 1). To define the \textbf{parity} of a run $A$, 
% denoted $\mathrm{par}(A)$, 
we say that $A$ is \textbf{even} if it consists of 0's and \textbf{odd} if it consists of 1's. 
% Given a binary word $u$, denote by $u^e$ the last even run of $u$, and denote by $u^o$ the last odd run of $u$.

\begin{exm}
We write the bolded subword $w'$ of $w = 01\textbf{11}0\textbf{0}1\textbf{101}001\textbf{10}$ as $w' = (3,4,6,8,9,$ $10,14,15)$. We have $|w| = 15$ and $|w'| = 8$. $w$ has 5 even runs and 4 odd runs. For the pattern $p = 010011010110$, we have $c_p(w) = 3 \cdot 2 = 6$.
\end{exm}

\section{Formula for computing $c_p(w)$}
\label{calc occur}
Given binary words $p = p_1p_2 \ldots p_m$ and $w = w_1w_2 \ldots w_n$, we will compute $c_p(w)$. We introduce a method that we call \textbf{lexicographic extreme referencing}. Roughly speaking, we will use the lexicographically minimal occurrence of $p$ in $w$, if it exists, as a reference point to aid us in counting all occurrences.

Suppose $p$ consists of $r$ runs with \[p = A_1A_2 \ldots A_r,\] where \[A_i = p_{a^i_1}p_{a^i_2} \ldots p_{a^i_{c_i}}.\] Suppose $w$ consists of $s$ runs with \[w = B_1B_2 \ldots B_s,\] where \[B_i = w_{b^i_1}w_{b^i_2} \ldots w_{b^i_{d_i}}.\] For any letter $w_k$ of $w$, let $\mathrm{R}(w,w_k) \in [s]$ be the index of the run $B_{\mathrm{R}(w,w_k)}$ containing $w_k$.

For two runs $B_i, B_j$ of the same parity, we define the \textbf{same-parity sum from $B_i$ to $B_j$ in $w$}, denoted $\sigma(w;i,j)$, as $$\sigma(w;i,j) : = \sum_{k=0}^{\frac{j-i}{2}}{|B_{i+2k}|}.$$ We define the \textbf{same-parity left-span by $j$ units from $B_i$ in $w$}, denoted $\lambda(w;i;j)$, as $$\lambda(w;i;j) := \min\{k \mid \sigma(w;i-2k,i) \geq j\}.$$ Analogously, define the \textbf{same-parity right-span by $j$ units from $B_i$ in $w$}, denoted $\rho(w;i;j)$, as $$\rho(w;i;j) := \min\{k \mid \sigma(w;i,i+2k) \geq j\}.$$ 

For any run $B_i = w_{b^i_1}w_{b^i_2} \ldots w_{b^i_{d_i}}$ of $w$, let $\phi(B_i) := b^i_1$ and $\psi(B_i) = b^i_{d_i}$. For any two same-parity letters $w_{l_1}$ and $w_{l_2}$, let $\delta_w(l_1,l_2)$ be the number of same-parity letters between $w_{l_1}$ and $w_{l_2}$ inclusive.

We will match $A_1$ with the first run of $w$ having the same parity. Define $$\alpha_p^w :=  
     \begin{cases}
       1 &\quad\text{if }A_1 \text{ and }B_1 \text{ have the same parity } \\
       2 &\quad\text{otherwise}\\
     \end{cases}.$$

Similarly, we will match $A_r$ with the last run of $w$ having the same parity. Define $$\zeta_p^w :=   
     \begin{cases}
       s &\quad\text{if }A_r \text{ and }B_s \text{ have the same parity } \\
       s-1 &\quad\text{otherwise}\\
     \end{cases}.$$

\begin{exm}
Let $p = 110100$ and $w = 011100110100110$. We have $\mathrm{R}(w,w_3) = 2$, $\sigma(w;4,8) = 2+1+2 = 5$, $\lambda(w;7;4) = 2$, $\rho(w;2;6) = 2$, $\alpha_p^w = 2$, and $\zeta_p^w = 9$. We have $\phi(B_2) = 2$, $\psi(B_4) = 8$, and $\delta_w(3,13) = 6$.
\end{exm}

Denote by $\mathrm{L}_p(w)$ the lexicographic minimal occurrence of $p$ in $w$, if it exists; this means that, for any occurrence $p'$ in $w$ and any $j \in [m]$, the $j$th letter of $p'$ cannot be left of the $j$th letter of $\mathrm{L}_p(w)$ in $w$. We have $\mathrm{L}_p(w) = L_1L_2 \ldots L_r$, where $L_i = w_{l^i_1}w_{l^i_2} \ldots w_{l^i_{k_i}}$ is equal to $A_i$ as runs. Also, we define the \textbf{$i$th head of $\mathrm{L}_p(w)$} as $\mathrm{L}_p^i(w) := \mathrm{R}(w,w_{l_1^i})$; in other words, $\mathrm{L}_p^i(w)$ is the index of the first run of $w$ that intersects $L_i$. 
\begin{rem}
Note that $\mathrm{L}_p(w)$ is completely determined by the indices $\mathrm{L}_p^i(w)$, which will play a crucial role in our formula for $c_p(w)$ later.
\end{rem}
\begin{lem}
\label{lex min def}
Let $p = p_1p_2 \ldots p_m$ and $w = w_1w_2 \ldots w_n$ be binary words. Suppose $p$ consists of $r$ runs with $p = A_1A_2 \ldots A_r$, where $A_i = p_{a^i_1}p_{a^i_2} \ldots p_{a^i_{c_i}}$. Suppose $w$ consists of $s$ runs with $w = B_1B_2 \ldots B_s$, where $B_i = w_{b^i_1}w_{b^i_2} \ldots w_{b^i_{d_i}}$. If $\mathrm{L}_p(w)$ exists, then the indices $\mathrm{L}_p^i(w)$ are determined as follows:
\begin{enumerate}
\item $\mathrm{L}_p^1(w) = \alpha_p^w$
% \item $\mathrm{L}_p^2(w) = \rho(w;\mathrm{L}_p^1(w);|A_1|)+1$
\item In general, we have $\mathrm{L}_p^{i+1}(w) = \mathrm{L}_p^i(w)+2\rho(w;\mathrm{L}_p^i(w);|A_i|)+1$
\end{enumerate}
\end{lem}
\begin{proof}
Since $B_{\mathrm{L}_p^1(w)}$ is the first run of $w$ with the same parity as $A_1$, we have $\mathrm{L}_p^1(w) = \alpha_p^w$. Since $B_{\mathrm{L}_p^2(w)}$ is the run immediately after the rightmost run spanned by $|A_1|$ units from $\mathrm{L}_p^1(w)$, we have $\mathrm{L}_p^2(w) = \mathrm{L}_p^1(w)+2\rho(w;\mathrm{L}_p^1(w);|A_1|)+1$. In general, since $B_{\mathrm{L}_p^{i+1}(w)}$ is the run immediately after the rightmost run spanned by $|A_i|$ units from $\mathrm{L}_p^{i}(w)$, we have $\mathrm{L}_p^{i+1}(w) = \mathrm{L}_p^i(w)+2\rho(w;\mathrm{L}_p^i(w);|A_i|)+1$.
\end{proof}

\begin{exm}
Let $p = 110100$ and $w = 011100110100110$. The lexicographic minimal occurrence $\mathrm{L}_p(w)$ is shown in bold below: \[w = 0\textbf{11}1\textbf{0}0\textbf{1}1\textbf{0}1\textbf{0}0110\]
$\mathrm{L}_p(w)$ corresponds to the indices $(2,3,5,7,9,11)$. We have $\mathrm{L}_p^1(w)=2$, $\mathrm{L}_p^2(w)=3$, $\mathrm{L}_p^3(w)=4$ $\mathrm{L}_p^4(w)=5$. 
\end{exm}

\begin{exm}
Let $p = 0011100011$ and $w = 1001011101000111010$. The lexicographic minimal occurrence $\mathrm{L}_p(w)$ is shown in bold below: \[w = 1\textbf{00}\textbf{1}0\textbf{11}1\textbf{0}1\textbf{00}0\textbf{11}1010\]
$\mathrm{L}_p(w)$ corresponds to the indices $(2,3,4,6,7,9,11,12,14,15)$. We have $\mathrm{L}_p^1(w) = 2$, $\mathrm{L}_p^2(w) = 3$, $\mathrm{L}_p^3(w) = 6$, $\mathrm{L}_p^4(w) = 9$.
\end{exm}

Our formula for $c_p(w)$ consists of $r$ nested summations, one for each run of $p$. It expresses $c_p(w)$ in terms of the run-structure of $p$ and $w$ as well as the indices $\mathrm{L}_p^i(w)$, which by Lemma \ref{lex min def} are also determined by the run-structure. The rough idea of this formula is that all occurrences in $w$ are sandwiched between the lexicographic minimum and the lexicographic maximum.
\begin{thm}
\label{count form}
Let $p = p_1p_2 \ldots p_m$ and $w = w_1w_2 \ldots w_n$ be binary words. Suppose $p$ consists of $r$ runs with $p = A_1A_2 \ldots A_r$, where $A_i = p_{a^i_1}p_{a^i_2} \ldots p_{a^i_{c_i}}$. Suppose $w$ consists of $s$ runs with $w = B_1B_2 \ldots B_s$, where $B_i = w_{b^i_1}w_{b^i_2} \ldots w_{b^i_{d_i}}$. If $\mathrm{L}_p(w)$ exists, then we have 
\begin{multline}
c_p(w) = \sum_{\substack{
q_1=|A_r|, |A_r|+1,\dots ,\sigma(w;\mathrm{L}_p^r(w),\zeta_p^w)\\
q_2=|A_{r-1}|,|A_{r-1}|+1,\ldots,\sigma(w;\mathrm{L}_p^{r-1}(w),t(q_1))\\
q_3=|A_{r-2}|,|A_{r-2}|+1,\ldots,\sigma(w;\mathrm{L}_p^{r-2}(w),t(q_2))\\
\vdots\\
q_r=|A_1|,|A_1|+1,\ldots,\sigma(w;\mathrm{L}_p^1(w),t(q_{r-1}))
}} {q_1-1 \choose |A_r|-1}{q_2-1 \choose |A_{r-1}|-1}\ldots{q_r-1 \choose |A_1|-1} \\
\end{multline}
where $t(q_1) = \zeta_p^w-2\lambda(w;\zeta_p^w;q_1)-1$ and $t(q_{k+1}) = t(q_k)-2\lambda(w;t(q_k);q_{k+1})-1$.
% Need to change $t(q_k)$ to $t_p^w(q_k)$?
\end{thm}
\begin{proof}
Note that $B_{t(q_{k+1})}$ is the run immediately left of the leftmost run spanned by $q_{k+1}$ units from $B_{t(q_{k})}$ in $w$, by definition. Let $p' = A'_1A'_2 \ldots A'_r$ be an occurrence of $p$ in $w$, where $A'_i$ is equal to $A_i$ as runs. Then the runs $A'_r, A'_{r-1}, \ldots, A'_1$ are determined in the following manner:
\begin{enumerate}
\item Since $A'_r$ must be no further left than $B_{\mathrm{L}_p^r(w)}$ by lexicographic minimality, choose $|A_r|$ letters from among the runs $B_{\zeta_p^w}, B_{\zeta_p^w-2}, \ldots, B_{\mathrm{L}_p^r(w)+2}, B_{\mathrm{L}_p^r(w)}$ to form $A'_r$. If $w_{r'_1}$ is the leftmost of these letters, then $q_1 := \delta_w(r'_1,\psi(B_{\zeta_p^w}))$ ranges through $q_1 = |A_r|, |A_r|+1, \ldots, \sigma(w;\mathrm{L}_p^r(w),\zeta_p^w)$.
\item Since $A'_{r-1}$ must be to the left of $A'_r$ and no further left than $B_{\mathrm{L}_p^{r-1}(w)}$ by lexicographic minimality, choose $|A_{r-1}|$ letters from among the runs $B_{t(q_1)}, B_{t(q_1)-2}, \ldots, B_{\mathrm{L}_p^{r-1}(w)+2}, B_{\mathrm{L}_p^{r-1}(w)}$ to form $A'_{r-1}$. If $w_{r'_2}$ is the leftmost of these letters, then $q_2 := \delta_w(r'_2,\psi(B_{t(q_1)}))$ ranges through $q_2 = |A_{r-1}|, |A_{r-1}|+1, \ldots, \sigma(w;\mathrm{L}_p^{r-1}(w),t(q_1))$.
\item In general, suppose $A'_{r-i+1}$ has been determined with associated $q_i$, for some $i < r$. We now determine $A'_{r-i}$. Since $A'_{r-i}$ must be to the left of $A'_{r-i+1}$ and no further left than $B_{\mathrm{L}_p^{r-i}(w)}$ by lexicographic minimality, choose $|A_{r-i}|$ letters from among the runs $B_{t(q_i)}, B_{t(q_i)-2}, \ldots, B_{\mathrm{L}_p^{r-i}(w)+2}, B_{\mathrm{L}_p^{r-i}(w)}$ to form $A'_{r-i}$. If $w_{r'_{i+1}}$ is the leftmost of these letters, then $q_{i+1} := \delta_w(r'_{i+1},\psi(B_{t(q_i)}))$ ranges through $q_{i+1} = |A_{r-i}|, |A_{r-i}|+1, \ldots, \sigma(w;\mathrm{L}_p^{r-i}(w),t(q_i))$.
\end{enumerate}
It follows that 
\begin{equation*}\begin{split}
c_p(w) = \sum_{q_1=|A_r|}^{\sigma(w;\mathrm{L}_p^r(w),\zeta_p^w)}{q_1-1 \choose |A_r|-1}\sum_{q_2=|A_{r-1}|}^{\sigma(w;\mathrm{L}_p^{r-1}(w),t(q_1))}{q_2-1 \choose |A_{r-1}|-1}
\sum_{q_3=|A_{r-2}|}^{\sigma(w;\mathrm{L}_p^{r-2}(w),t(q_2))}{q_3-1 \choose |A_{r-2}|-1}\ldots \\
\sum_{q_r=|A_1|}^{\sigma(w;\mathrm{L}_p^1(w),t(q_{r-1}))}{q_r-1 \choose |A_1|-1}                    
\end{split}\end{equation*}
and hence 
\begin{multline*}
c_p(w) = \sum_{\substack{
q_1=|A_r|, |A_r|+1,\dots ,\sigma(w;\mathrm{L}_p^r(w),\zeta_p^w)\\
q_2=|A_{r-1}|,|A_{r-1}|+1,\ldots,\sigma(w;\mathrm{L}_p^{r-1}(w),t(q_1))\\
q_3=|A_{r-2}|,|A_{r-2}|+1,\ldots,\sigma(w;\mathrm{L}_p^{r-2}(w),t(q_2))\\
\vdots\\
q_r=|A_1|,|A_1|+1,\ldots,\sigma(w;\mathrm{L}_p^1(w),t(q_{r-1}))
}} {q_1-1 \choose |A_r|-1}{q_2-1 \choose |A_{r-1}|-1}\ldots{q_r-1 \choose |A_1|-1} \\
\end{multline*}
\end{proof}

% \begin{exm}
% Let $p = 110100$ and $w = 011100110100110$. Given 
% \end{exm}

\begin{exm}
Let $p = 0011100011$ and $w = 1001011101000111010$. We have $r=4$ and $\zeta_p^w = 11$, and we have $\mathrm{L}_p^1(w) = 2$, $\mathrm{L}_p^2(w) = 3$, $\mathrm{L}_p^3(w) = 6$, $\mathrm{L}_p^4(w) = 9$. 
\begin{enumerate}
\item Since $\sigma(w;9,11) = 4$, we have $2 \leq q_1 \leq 4$.
\item If $q_1 = 3$, then $t(q_1) = 11-2\lambda(w;11;3)-1 = 8$ and $\sigma(w;6,8) = 4$, so we have $3 \leq q_2 \leq 4$.
\item If $q_2 = 4$, then $t(q_2) = 8-2\lambda(w;8;4)-1 = 5$ and $\sigma(w;3,5) = 4$, so we have $3 \leq q_3 \leq 4$.
\item If $q_3 = 4$, then $t(q_3) = 5-2\lambda(w;5;4)-1 = 2$ and $\sigma(w;2,2) = 2$, so we have $2 \leq q_4 \leq 2$ $\Rightarrow$ $q_4 = 2$.
\end{enumerate}
\end{exm}

\begin{exm}
Let $p = 1111000001$ and $w = 101010011001010101011$. We have $r=3$ and $\zeta_p^w = 17$, and we have $\mathrm{L}_p^1(w) = 1$, $\mathrm{L}_p^2(w) = 8$, $\mathrm{L}_p^3(w) = 15$.
\begin{enumerate}
\item Since $\sigma(w;15,17) = 3$, we have $1 \leq q_1 \leq 3$.
\item If $q_1 = 2$, then $t(q_1) = 17-2\lambda(w;17;2)-1 = 16$ and $\sigma(w;8,16) = 6$, so we have $5 \leq q_2 \leq 6$.
\item If $q_2 = 6$, then $t(q_2) = 16-2\lambda(w;16;6)-1 = 7$ and $\sigma(w;1,7) = 5$, so we have $4 \leq q_3 \leq 5$.
\end{enumerate}
\end{exm}

\begin{exm}
Let $p = 0010$ and $w = 0001010$. Then 
\begin{align*}
c_p(w) &= \sum_{q_1=1}^{2}{\sum_{q_2=1}^{\sigma(w;2,t(q_1))}{\sum_{q_3=2}^{\sigma(w;1,t(q_2))}{{q_1-1 \choose 0}{q_2-1 \choose 0}{q_3-1 \choose 1}}}} \\
&= \left(\sum_{q_2=1}^{\sigma(w;2,4)}{\sum_{q_3=2}^{\sigma(w;1,t(q_2))}{q_3-1}}\right) + \left(\sum_{q_2=1}^{\sigma(w;2,2)}{\sum_{q_3=2}^{\sigma(w;1,t(q_2))}{q_3-1}}\right) \\
&= \left(\sum_{q_3=2}^{\sigma(w;1,3)}{q_3-1}\right) + \left(\sum_{q_3=2}^{\sigma(w;1,1)}{q_3-1}\right) + \left(\sum_{q_3=2}^{\sigma(w;1,1)}{q_3-1}\right) \\ 
&= 6+3+3 \\
&= 12
\end{align*}
\end{exm}

\section{Application to finding $B_{n,p}(k)$}
Let ${\cal B}_{n,p}(k)$ denote the set of words of length $n$ containing $p$ exactly $k$ times, so $B_{n,p}(k) = |{\cal B}_{n,p}(k)|$. In the absence of a general formula for computing the numbers $B_{n,p}(k)$ for arbitrary $n$ and $k$, one could opt for a brute-force approach that involves searching for the elements of ${\cal B}_{n,p}(k)$ in a structured way. We first give an algorithm for building all words of length $n$ containing $p$. Then we provide a modified version of this algorithm that constructs only those words that are ``primitive'', streamlining the search process.
\subsection{Building binary word from given pattern}
\label{build word}
If $w$ is a binary word with length $n$ containing $p$, then we can construct $w$ by adding $n-m$ letters to $p$ in a manner that preserves $p$ as the lexicographically minimal occurrence. To do this, we first characterize lexicographic minimality as follows.
\begin{lem}
\label{lex char}
Let $p = p_1p_2 \ldots p_m$ be a binary word consisting of $r$ runs with $p = A_1A_2 \ldots A_r$, where $A_i = p_{a^i_1}p_{a^i_2} \ldots p_{a^i_{c_i}}$. Suppose $p' = A'_1A'_2 \ldots A'_r$ is an occurrence of $p$ in $w$, where $A'_i$ equals $A_i$ as runs. Then $p'$ is lexicographically minimal if and only if the following holds:

For every run $A'_i$, if a letter $p_{a^i_1}$ occurs before $A'_i$ in $w$, then there are no occurrences of the subword $A_1A_2 \ldots A_{i-1}$ before this letter.
\end{lem}

Lemma \ref{lex char} imposes restrictions on where the $n-m$ additional letters can be inserted in $p$. Taking this into account, we now provide the following algorithm to construct every word of a given length containing $p$. Roughly speaking, we insert letters for every run of $p$ in reverse order, and for each run $A_i = p_{a^i_1}p_{a^i_2} \ldots p_{a^i_{c_i}}$ of $p$, we insert letters $p_{a^i_1}$ after $A_i$ but not after $A_{i+1}$.

\begin{thm}
\label{algo gen}
% Let $\eta_i^j(w)$ denote the sub-run consisting of the first $j$ letters of $A_i$.
Let $p = p_1p_2 \ldots p_m$ be a binary word consisting of $r$ runs with $p = A_1A_2 \ldots A_r$, where $A_i = p_{a^i_1}p_{a^i_2} \ldots p_{a^i_{c_i}}$. Suppose $w$ has length $n$ and contains $p$. Then $w$ is constructed from $p$ as follows. Choose nonnegative integers $s_1, s_2, \ldots, s_r, s_{r+1}$ such that $s_1 + s_2 + \ldots + s_r + s_{r+1} = n-m$.

\begin{enumerate}
\item Construct the word $\chi_1(p) := A_1A_2 \ldots A_{r-1}X_r$, where $X_r$ is obtained by adding $s_1$ copies of $p_{a_1^{r}}$ after $A_r$.
\item Construct the word $\chi_2(p) := A_1A_2 \ldots A_{r-1}X_{r-1}$, where $X_{r-1}$ is obtained by permuting the letters of $X_r$ and $s_2$ copies of $p_{a_1^{r-1}}$.
\item Construct the word $\chi_3(p) := A_1A_2 \ldots A_{r-2}X_{r-2}X_{r-1}$, where $X_{r-2} = X'_{r-2}p_{a_{c_{r-1}}^{r-1}}$ and $X'_{r-2}$ is obtained by permuting the first $|A_{r-1}|-1$ letters of $A_{r-1}$ and $s_3$ copies of $p_{a_1^{r-2}}$.
\item In general, construct the word $\chi_k(p) := A_1A_2 \ldots A_{r-k+1}X_{r-k+1}X_{r-k+2} \ldots X_{r-1}$, where $X_{r-k+1} = X'_{r-k+1}p_{a_{c_{r-k+2}}^{r-k+2}}$ and $X'_{r-k+1}$ is obtained by permuting the first $|A_{r-k+2}|-1$ letters of $A_{r-k+2}$ and $s_k$ copies of $p_{a_1^{r-k+1}}$.
\item Finally, construct the word $w = \chi_{r+1}(p) := X_0X_1 \ldots X_{r-1}$, where $X_0 = X'_0p_{a_{c_1}^1}$ and $X'_0$ is obtained by permuting the first $|A_1|-1$ letters of $A_1$ and $s_{r+1}$ copies of $1-p_{a_1^1}$.
\end{enumerate}
\end{thm}

In the examples that follow, for each letter $v$, we use $\bar{v}$ to mark $v$ as part of a run of the original word $p$ and we use $\hat{v}$ to mark $v$ as a newly added letter.
\begin{exm}
Let $p = 0010000111011$. We can construct a word $w$ of length 26 as follows.
\begin{enumerate}
\item $\chi_1(p) = 00100001110\bar{1}\bar{1}$ (here we chose to add no additional letters)
\item $\chi_2(p) = 0010000111\bar{0}\hat{0}1\hat{0}1$
\item $\chi_3(p) = 0010000\bar{1}\bar{1}\bar{1}\hat{1}00101$
\item $\chi_4(p) = 001\bar{0}\bar{0}\bar{0}\bar{0}1\hat{0}1\hat{0}\hat{0}1100101$
\item $\chi_5(p) = 00\bar{1}0\hat{1}0\hat{1}\hat{1}0\hat{1}0101001100101$
\item $\chi_6(p) = \bar{0}\bar{0}\hat{0}101011010101001100101$
\item $w = \chi_7(p) = \hat{1}0\hat{1}00101011010101001100101$
\end{enumerate}
\end{exm}

\subsection{Primitive words}
\label{primitive}
Let $p = p_1p_2 \ldots p_m$ and $w = w_1w_2 \ldots w_n$ be binary words. Suppose $p$ consists of $r$ runs with $p = A_1A_2 \ldots A_r$, where $A_i = p_{a^i_1}p_{a^i_2} \ldots p_{a^i_{c_i}}$. Suppose $w$ consists of $s$ runs with $w = B_1B_2 \ldots B_s$, where $B_i = w_{b^i_1}w_{b^i_2} \ldots w_{b^i_{d_i}}$. 

For a positive integer $k$ and $j \in [n]$, let $\mathrm{Llet}_w(j,k)$ denote the index of the $(k-1)$st same-parity letter left of $w_j$; in other words, there are exactly $k$ same-parity letters between $w_j$ and $w_{\mathrm{Llet}_w(j,k)}$ inclusive. Analogously, let $\mathrm{Rlet}_w(j,k)$ denote the index of the $(k-1)$st same-parity letter right of $w_j$. Let $\mathrm{tail}(w;j)$ denote the subword $w_jw_{j+1} \ldots w_n$.
\begin{exm}
For $w = 10100101011010101001100101$, we have $\mathrm{Llet}_w(18,7) = 5$, $\mathrm{Rlet}_w(11,4) = 17$, and $\mathrm{tail}(w;19) = 01100101$. 
\end{exm}

In analogy to $\mathrm{L}_p(w)$, we define (cf. Lemma \ref{lex min def}) the lexicographic maximal occurrence $\mathrm{R}_p(w)$ of $p$ in $w$ and the associated indices $\mathrm{R}_p^i(w)$ as follows:
\begin{enumerate}
\item $\mathrm{R}^1_p(w) = \zeta_p^w$
\item In general, we have $\mathrm{R}^{i+1}_p(w) = \mathrm{R}^i_p(w) - 2\lambda(w;\mathrm{R}^i_p(w);|A_{r-i+1}|)-1$
\end{enumerate}
We will use the shorthand $\psi_p^i[w] := \psi(B_{\mathrm{R}^{i}_p(w)})$, denoting the index of the last letter of $B_{\mathrm{R}^{i}_p(w)}$.
\begin{exm}
Let $p = 0010000111011$ and $w = 10100101011010101001100101$. Then we have the following:
\begin{enumerate}
\item $\mathrm{R}_p^1(w) = 21$ and $\psi_p^1[w] = 26$
\item $\mathrm{R}_p^2(w) = 18$ and $\psi_p^2[w] = 23$
\item $\mathrm{R}_p^3(w) = 17$ and $\psi_p^3[w] = 21$
\item $\mathrm{R}_p^4(w) = 14$ and $\psi_p^4[w] = 16$
\item $\mathrm{R}_p^5(w) = 7$ and $\psi_p^5[w] = 8$
\item $\mathrm{R}_p^6(w) = 6$ and $\psi_p^6[w] = 7$
\end{enumerate}
\end{exm}

We call $w$ \textbf{primitive with respect to $p$} if every letter of $w$ is contained in an occurrence of $p$. Note that if $w = B_1B_2 \ldots B_s$ is primitive, then $\alpha_p^w = 1$ and $\zeta_p^w = s$. 

\begin{exm}
For $p = 0010000111011$, we have the following: 
\begin{enumerate}
\item $w = 0\hat{1}0100\hat{1}0\hat{1}01\hat{0}\hat{0}1101\hat{0}1$ is not primitive because the hatted letters are not contained in any occurrence of $p$.
\item $w = 00010010\hat{1}00101110011$ is not primitive because the hatted letter is not contained in any occurrence of $p$.
\item $w = 0001001000101110011$ is primitive.
\end{enumerate}
\end{exm}

We characterize the primitive words in the following
\begin{thm}
\label{ext let}
Let $p = p_1p_2 \ldots p_m$ and $w = w_1w_2 \ldots w_n$ be binary words. Suppose $p$ consists of $r$ runs with $p = A_1A_2 \ldots A_r$, where $A_i = p_{a^i_1}p_{a^i_2} \ldots p_{a^i_{c_i}}$. Suppose $w$ consists of $s$ runs with $w = B_1B_2 \ldots B_s$, where $B_i = w_{b^i_1}w_{b^i_2} \ldots w_{b^i_{d_i}}$. For any $l \in [n]$, if letter $w_l$ does not belong to $\mathrm{L}_p(w)$ or $\mathrm{R}_p(w)$, then the following hold.
\begin{enumerate}
\item If $\mathrm{Llet}_w(\psi_p^j[w],|A_{r-j+1}|) < l < \mathrm{Rlet}_w(\phi(B_{\mathrm{L}_p^{r-j+1}(w)}),|A_{r-j+1}|)$ for some $j \in [r]$, then $w_l$ is not contained in any occurrence of $p$.
\item If $l < \mathrm{Rlet}_w(\phi(B_{\mathrm{L}_p^{1}(w)}),|A_{1}|)$ or $l > \mathrm{Llet}_w(\psi_p^1[w],|A_{r}|)$, then $w_l$ is not contained in any occurrence of $p$. 
\item Otherwise, $w_l$ is contained in an occurrence of $p$ in $w$.
\end{enumerate}
\end{thm}
\begin{proof}
First note that $w_l$ is automatically contained in an occurrence of $p$ if it belongs to $\mathrm{L}_p(w)$ or $\mathrm{R}_p(w)$. We now handle the aforementioned three cases separately; see Example \ref{lem ex} for illustration.
\begin{enumerate}
\item We need only consider the case $w_l = 1-w_{\psi_p^j[w]}$. Suppose $p'$ is an occurrence of $p$ containing $w_l$. Since $\mathrm{Llet}_w(\psi_p^j[w],|A_{r-j+1}|) < l < \mathrm{Rlet}_w(\phi(B_{\mathrm{L}_p^{r-j+1}(w)}),|A_{r-j+1}|)$, $p'$ must be either lexicographically less than $\mathrm{L}_p(w)$ or lexicographically greater than $\mathrm{R}_p(w)$, which is a contradiction.
\item We need only consider the case $l < \mathrm{Rlet}_w(\phi(B_{\mathrm{L}_p^{1}(w)}),|A_{1}|)$ with $w_l = 1-w_{\phi(B_{\mathrm{L}_p^{1}(w)})}$. Any occurrence of $p$ containing $w_l$ would then need to have at least $|A_1|$ letters $w_{\phi(B_{\mathrm{L}_p^{1}(w)})}$ left of $w_l$, which is impossible. The situation $l > \mathrm{Llet}_w(\psi_p^1[w],|A_{r}|)$ is analogous.
\item We need only consider the case $w_l = 1-w_{\psi_p^j[w]}$. If $\mathrm{Llet}_w(\psi_p^j[w],|A_{r-j+1}|) > l$ and $l > \phi(B_{\mathrm{L}_p^{r-j+1}(w)})$, then an occurrence of $p$ containing $w_l$ can be constructed by concatenating a subword of $\mathrm{L}_p(w)$, $w_l$, and a subword of $\mathrm{R}_p(w)$, in that order. The occurrence of $p$ can be constructed analogously in the situation $l > \mathrm{Rlet}_w(\phi(B_{\mathrm{L}_p^{r-j+1}(w)}),|A_{r-j+1}|)$ and $l < \psi_p^j[w]$.
\end{enumerate}
\end{proof}
\begin{exm}
\label{lem ex}
For $p = 0010000111011$, observe the following.
\begin{enumerate}
\item If $w = 1010010\hat{1}0\hat{1}\hat{1}0\hat{1}0101001100101$, the hatted letters are not contained in any occurrence of $p$, because otherwise such an occurrence would have to be either lexicographically less than $\mathrm{L}_p(w)$ or lexicographically greater than $\mathrm{R}_p(w)$.
\item If $w = 10\hat{1}00101011010101001100101$, the hatted letter is not contained in any occurrence of $p$ because there are not enough 0's to the left.
\item If $w = \textbf{0010000}1101\hat{\textbf{1}}0001\textbf{11011}$, the hatted letter is contained in the bolded subword.
\item If $w = \textbf{00}100\hat{\textbf{1}}0011\textbf{0}1\textbf{000}\textbf{111011}$, the hatted letter is contained in the bolded subword.
\end{enumerate}
\end{exm}

Using Theorem \ref{ext let}, we can construct primitive words containing $p$ by the following modified form of the general procedure in Theorem \ref{algo gen}. Roughly speaking, an additional restriction on the insertion of $n-m$ letters in $p$ is imposed: At the $k$th stage $\eta_k(p)$ of the construction, no letters can be inserted between the $(\mathrm{Llet}_{\eta_{k-1}(p)}(\psi_p^{k-1}[\eta_{k-1}(p)],|A_{r-k+2}|))$th and the $(\psi(A_{r-k+2}))$th letters of $\eta_{k-1}(p)$.

\begin{thm}
\label{algo prim}
Let $p = p_1p_2 \ldots p_m$ be a binary word consisting of $r$ runs with $p = A_1A_2 \ldots A_r$, where $A_i = p_{a^i_1}p_{a^i_2} \ldots p_{a^i_{c_i}}$. Suppose $w$ has length $n$ and contains $p$. Then $w$ is primitive with respect to $p$ if and only if $w$ is constructed in the following manner:

Choose nonnegative integers $t_1, t_2, \ldots, t_r$ such that $t_1 + t_2 + \ldots + t_r = n-m$.
\begin{enumerate}
\item Construct the word $\eta_1(p) := A_1A_2 \ldots A_{r-1}Y_r$, where $Y_r$ is obtained by adding $t_1$ copies of $p_{a_1^{r}}$ after $A_r$.
\item  Construct the word $\eta_2(p) := A_1A_2 \ldots A_{r-1}Y_{r-1}Y'_{r-1}$, where \[Y'_{r-1} = \mathrm{tail}(\eta_1(p),\min(\mathrm{Llet}_{\eta_1(p)}(\psi_p^1[\eta_1(p)],|A_r|),\psi(A_r)))\] and $Y_{r-1}$ is obtained by permuting $\delta_{\eta_1(p)}(\phi(A_r),\min(\mathrm{Llet}_{\eta_1(p)}(\psi_p^1[\eta_1(p)],|A_r|),\psi(A_r)))-1$ copies of $p_{a_1^r}$ and $t_2$ copies of $p_{a_1^{r-1}}$.
\item Construct the word $\eta_3(p) := A_1A_2 \ldots A_{r-2}Y_{r-2}Y'_{r-2}$, where \[Y'_{r-2} = \mathrm{tail}(\eta_2(p),\min(\mathrm{Llet}_{\eta_2(p)}(\psi_p^2[\eta_2(p)],|A_{r-1}|),\psi(A_{r-1}))) \] and $Y_{r-2}$ is obtained by permuting $$\delta_{\eta_2(p)}(\phi(A_{r-1}),\min(\mathrm{Llet}_{\eta_2(p)}(\psi_p^2[\eta_2(p)],|A_{r-1}|),\psi(A_{r-1})))-1$$ copies of $p_{a_1^{r-1}}$ and $t_3$ copies of $p_{a_1^{r-2}}$.
\item In general, construct the word $\eta_k(p) := A_1A_2 \ldots A_{r-k+1}Y_{r-k+1}Y'_{r-k+1}$, where \[Y'_{r-k+1} = \mathrm{tail}(\eta_{k-1}(p),\min(\mathrm{Llet}_{\eta_{k-1}(p)}(\psi_p^{k-1}[\eta_{k-1}(p)],|A_{r-k+2}|),\psi(A_{r-k+2}))) \] and $Y_{r-k+1}$ is obtained by permuting $$\delta_{\eta_{k-1}(p)}(\phi(A_{r-k+2}),\min(\mathrm{Llet}_{\eta_{k-1}(p)}(\psi_p^{k-1}[\eta_{k-1}(p)],|A_{r-k+2}|),\psi(A_{r-k+2})))-1$$ copies of $p_{a_1^{r-k+2}}$ and $t_k$ copies of $p_{a_1^{r-k+1}}$.
\item Finally, take $w := \eta_r(p)$.
\end{enumerate}
\end{thm}

\begin{exm}
Let $p = 110011100011$. We can construct a primitive word $w$ of length 17 as follows.
\begin{enumerate}
\item $\eta_1(p) = 1100111000\bar{1}\bar{1}\hat{1}$
\item $\eta_2(p) = 1100111\bar{0}\bar{0}\bar{0}1\hat{0}11$ (note that no more 0's can be inserted after ``$\hat{0}1$'')
\item $\eta_3(p) = 1100\bar{1}\bar{1}\bar{1}0\hat{1}001011$ (note that no more 1's can be inserted after ``$\hat{1}$0'')
\item $\eta_4(p) = 11\bar{0}\bar{0}11101001011$ (here we chose to insert no 0's)
\item $w = \eta_5(p) = \bar{1}\bar{1}\hat{1}\hat{1}0011101001011$ (note that no more 1's can be inserted after ``$\hat{1}0$'', and that no more 0's can be inserted at all after this stage)
\end{enumerate}
\end{exm}

Clearly, any word containing $p$ can be reduced to its primitive form by removing the extraneous letters. Conversely, if $w$ is primitive with $c_p(w) = k > 0$, then $w$ is a minimal word containing $p$ exactly $k$ times, and the following procedure completes $w$ to a word $v$ of length $m' > n$ with $c_p(v) = k$.

\begin{prop}
\label{prim compl}
Let $p = p_1p_2 \ldots p_m$ be a binary word consisting of $r$ runs with $p = A_1A_2 \ldots A_r$, where $A_i = p_{a^i_1}p_{a^i_2} \ldots p_{a^i_{c_i}}$. Suppose $w$ is primitive with length $n$ and $c_p(w) = k > 0$. Then $v$ is a word of length $m' > n$ containing $w$ and $c_p(v) = k$ if and only if $v$ is constructed as follows. Choose nonnegative integers $u_1, u_2, \ldots, u_r$ such that $u_1+u_s+ \ldots + u_r = m'-n$. 
\begin{enumerate}
\item For each $j \in \{2,3,\ldots,r\}$, replace the subword \[w_{\mathrm{Llet}_w(\psi_p^j[w],|A_{r-j+1}|)+1} \ldots w_{\mathrm{Rlet}_w(\phi(B_{\mathrm{L}_p^{r-j+1}(w)}),|A_{r-j+1}|)-1}\] with \[w_{\mathrm{Llet}_w(\psi_p^j[w],|A_{r-j+1}|)+1} Z_j w_{\mathrm{Rlet}_w(\phi(B_{\mathrm{L}_p^{r-j+1}(w)}),|A_{r-j+1}|)-1},\] where $Z_j$ is obtained by permuting $\mathrm{Rlet}_w(\phi(B_{\mathrm{L}_p^{r-j+1}(w)}),|A_{r-j+1}|)-\mathrm{Llet}_w(\psi_p^j[w],|A_{r-j+1}|)$ $-3$ copies of $w_{\mathrm{Llet}_w(\psi_p^j[w],|A_{r-j+1}|)+1}$ and $u_j$ copies of $1-w_{\mathrm{Llet}_w(\psi_p^j[w],|A_{r-j+1}|)+1}$. 
% Note that clearly no letters $1-w_{\mathrm{Llet}_w(\psi_p^j[w],|A_{r-j+1}|)}$ can preexist in the subword $w_{\mathrm{Llet}_w(\psi_p^j[w],|A_{r-j+1}|)+1} \ldots w_{\mathrm{Rlet}_w(\phi(B_{\mathrm{L}_p^{r-j+1}(w)}),|A_{r-j+1}|)-1}$
\item Replace the subword $w_{\mathrm{Llet}_w(\psi_p^1[w],|A_{r}|)+1} \ldots w_{n}$ with $Z_1$, where $Z_1$ is obtained by permuting $|A_r|-1$ copies of $w_n$ and $u_1$ copies of $1-w_n$.
\item Replace the subword $w_{1} \ldots w_{\mathrm{Rlet}_w(\phi(B_{\mathrm{L}_p^{1}(w)}),|A_{1}|)-1}$ with $Z_r$, where $Z_r$ is obtained by permuting $|A_1|-1$ copies of $w_1$ and $u_r$ copies of $1-w_1$. 
\item $u_i = 0$ whenever $\mathrm{Llet}_w(\psi_p^j[w],|A_{r-j+1}|) > \mathrm{Rlet}_w(\phi(B_{\mathrm{L}_p^{r-j+1}(w)}),|A_{r-j+1}|)$.
\end{enumerate}
\end{prop}

\begin{exm}
Let $p = 0010000111011$. $w = 0010000\hat{0}111\hat{1}011$ is primitive with $k = c_{p}(w) = 20$. To complete $w$ to a word $v$ of length 20, additional letters must be inserted into the following spaces (indicated by blanks) in $w$: $$\_0\_0100\_0\_0011\_1101\_1\_$$
Some examples of $v$: $00100\hat{1}\hat{1}0\hat{1}0011\hat{0}11011\hat{0}$, $\hat{1}0\hat{1}010000011\hat{0}\hat{0}1101\hat{0}1$, $0\hat{1}\hat{1}0100000111101\hat{0}\hat{0}1\hat{0}$
\end{exm}

\begin{prop}
Let $p = p_1p_2 \ldots p_m$ and $w = w_1w_2 \ldots w_n$ be binary words. Suppose $p$ consists of $r$ runs with $p = A_1A_2 \ldots A_r$, where $A_i = p_{a^i_1}p_{a^i_2} \ldots p_{a^i_{c_i}}$. Suppose $w$ consists of $s$ runs with $w = B_1B_2 \ldots B_s$, where $B_i = w_{b^i_1}w_{b^i_2} \ldots w_{b^i_{d_i}}$. If $w$ is primitive with $c_p(w) = k > 0$, then there are 
\begin{equation*}
\begin{split}
PC(p,w,m') := \sum_{u_1+u_2+\ldots+u_r=m'-n}{|A_1|-1+u_r \choose u_r}{|A_r|-1+u_1 \choose u_1}\cdot \\
\prod_{j=2}^{r}{\mathrm{Rlet}_w(\phi(B_{\mathrm{L}_p^{r-j+1}(w)}),|A_{r-j+1}|)-\mathrm{Llet}_w(\psi_p^j[w],|A_{r-j+1}|)-3+u_j \choose u_j}
\end{split} \end{equation*} 
ways to complete $w$ to a word $v$ of length $m' > n$ with $c_p(v) = k$.
\end{prop}

Let $\mathrm{prim}_{n,p}(k)$ denote the set of all primitive words $w$ with $|w| \leq n$ and $c_p(w) = k$.

For arbitrary $n$ and $k$, we have 
\begin{equation}
B_{n,p}(k) = \sum_{w \in \mathrm{prim}_{n,p}(k)}PC(p,w,n).
\end{equation}
A brute-force algorithm to find the elements of $\mathrm{prim}_{n,p}(k)$ would be to use the procedure in Theorem \ref{algo prim} along with the formula in Theorem \ref{count form} to check whether $c_p(w) = k$.

Alternatively, notice that the solutions of the equation 
\begin{multline}
\label{prim eqn}
c_p(w) = \sum_{\substack{
q_1=|A_r|, |A_r|+1,\dots ,\sigma(w;\mathrm{L}_p^r(w),\zeta_p^w)\\
q_2=|A_{r-1}|,|A_{r-1}|+1,\ldots,\sigma(w;\mathrm{L}_p^{r-1}(w),t(q_1))\\
q_3=|A_{r-2}|,|A_{r-2}|+1,\ldots,\sigma(w;\mathrm{L}_p^{r-2}(w),t(q_2))\\
\vdots\\
q_r=|A_1|,|A_1|+1,\ldots,\sigma(w;\mathrm{L}_p^1(w),t(q_{r-1}))
}} {q_1-1 \choose |A_r|-1}{q_2-1 \choose |A_{r-1}|-1}\ldots{q_r-1 \choose |A_1|-1} = k \\
\end{multline}
are exactly the primitive words with $k$ occurrences; extraneous letters are not counted by the formula. One could try solving (\ref{prim eqn}) to find the elements of $\mathrm{prim}_{n,p}(k)$, in which case we take $\alpha_p^w = 1$ and $\zeta_p^w = s$.

\section{Future work}
Certainly, finding a way to compute $B_{n,p}(k)$ exactly for general $k$ is highly desired, if possible. Also, one can try to generalize the counting formula in Theorem \ref{count form} to words on a general alphabet $\{0,1,2,\ldots,l\}$ using similar concepts.

\end{document}